\documentclass{amsart}

\relax
\PassOptionsToClass{twoside}{article}

\RequirePackage{amsfonts,amssymb,amsmath,amscd,amsthm}
\RequirePackage{txfonts}
\RequirePackage{graphicx}
\RequirePackage{xcolor}
\RequirePackage{geometry}
\RequirePackage{enumerate}

\gdef\@linkcolor{cupred}
\gdef\@citecolor{cupred}
\newskip\belowtitleskip
\newskip\beforehistoryskip
\newskip\beforededicationskip
\newskip\aboveabsskip
\newskip\belowabsskip

\def\C{\mathbb C}
\def\R{\mathbb R}
\def\X{\mathbb X}

\def\U{\mathbb U}
\def\V{\mathbb U_B}
\def\W{\mathbb U_\infty}

\def\Y{\mathbb Y}

\newcommand{\om}{\omega}

\newcommand{\la}{\lambda}

\newcommand{\sub}{\subseteq}

\newtheorem{theorem}{Theorem}[section]
\newtheorem{lemma}[theorem]{Lemma}

\theoremstyle{definition}
\newtheorem{definition}[theorem]{Definition}
\newtheorem{example}[theorem]{Example}
\newtheorem{proposition}[theorem]{Proposition}
\newtheorem{corollary}[theorem]{Corollary}

\theoremstyle{remark}

\numberwithin{equation}{section}

\begin{document}

\title[Existence of Doubly-Weighted Pseudo Almost Periodic Solutions]{Existence of Doubly-Weighted Pseudo Almost Periodic Solutions to Some Classes of Nonautonomous \\Differential Equations}

\author{Toka Diagana}
\address{Department of Mathematics, Howard University, 2441 6th Street N.W., Washington, D.C. 20059, USA}
\email{tdiagana@howard.edu}


\subjclass[2000]{primary 35B15; secondary 34D09; 58D25; 42A75; 37L05}



\keywords{weight, weighted pseudo-almost periodic, doubly-weighted Bohr spectrum, almost periodic, doubly-weighted pseudo-almost
periodic.}

\begin{abstract} 
The main objective of this paper is twofold. We first show that if the doubly-weighted Bohr spectrum of an almost periodic function exists, then it is either empty or coincides with the Bohr spectrum of that function. Next, we investigate the problem which consists of the
existence of doubly-weighted pseudo-almost periodic solutions to some nonautonomous abstract differential equations. 
\end{abstract}

\maketitle

\section{Introduction}
Motivated by the functional structure of the so-called weighted Morrey spaces \cite{KA}, in Diagana \cite{DW}, a new concept called
doubly-weighted pseudo-almost periodicity, which generalizes in a natural fashion the notion of weighted pseudo-almost periodicity is introduced and studied.  
Among other things, in \cite{DW},  properties of these new functions have been studied including the stability of the convolution operator, the translation-invariance, the existence of a doubly-weighted mean for almost periodic functions under some reasonable assumptions, the uniqueness of the decomposition involving these new functions as well as some results on the composition of these new functions have been studied.

The main objective of this paper is twofold. We first show if the doubly-weighted Bohr spectrum of an almost periodic function exists, then it is either empty or coincides with the Bohr spectrum of that function. Next, we investigate the problem which consists of the
existence of doubly-weighted pseudo-almost periodic solutions to the nonautonomous abstract differential equations
\begin{eqnarray}\label{2}
u'(t) =
A(t) u(t) + g(t, u(t)),\; \; t \in \R,
\end{eqnarray}
where $A(t)$ for $t\in \R$ is a family of closed linear operators
on $D(A(t))$ satisfying the well-known Acquistapace and Terreni
conditions, and $g: \R \times \X \mapsto
\X$ is doubly weighted pseudo-almost periodic in $t \in \R$ uniformly in the
second variable. 

It is well-known that in that case, there exists
an evolution family ${\mathcal U} =\{U(t,s)\}_{t \geq s}$
associated with the family of linear operators $A(t)$. Assuming that the
evolution family ${\mathcal U} =\{U(t,s)\}_{t \geq s}$ is exponentially
dichotomic and under some additional assumptions it
will be shown that Eq. \eqref{2} has a unique doubly-weighted pseudo-almost
periodic solution.

The existence of weighted pseudo-almost periodic, weighted pseudo-almost automorphic,
and pseudo-almost periodic solutions to differential equations constitutes
one of the most attractive topics in qualitative theory of
differential equations due to possible applications. Some
contributions on weighted pseudo-almost periodic functions, their extensions, and their applications on differential equations
have recently been
made, among them are for instance \cite{AG1}, \cite{BB}, \cite{BE}, \cite{BE2}, \cite{DMN}, \cite{TO}, \cite{TOK}, \cite{DD}, \cite{liang},
\cite{L}, \cite{LS}, \cite{z4}, and \cite{LI} and the references therein.
However, the problem which consists of the
existence of doubly-weighted pseudo-almost periodic(mild) solutions to evolution
equations in the form Eq.~(\ref{2}) is quite new and
untreated and thus constitutes one of the main
motivations of the present paper.

The paper is organized as follows: Section 2 is devoted to
preliminaries results related to the existence of an evolution
family, intermediate spaces, properties of weights, and basic definitions and results on the
concept of doubly-weighted pseudo-almost periodic functions.
Section 3 is devoted to the existence of a doubly-weighted Bohr spectral theory for almost periodic functions while Section 4 is devoted to the existence of doubly-weighted pseudo-almost periodic solutions to Eq.~(\ref{2}).

\section{Preliminaries}
Let $(\X, \|\cdot\|)$ be a Banach space. If $C$ is a linear operator on $\X$, then
$D(C)$, $\rho(C)$, and $\sigma(C)$  stand respectively for the domain,
resolvent, and spectrum of $C$. Similarly, one sets $R(\lambda, C) := (\lambda I - C)^{-1}$ for all
$\la \in \rho(C)$ where $I$ is the identity operator for $\X$. 
Furthermore, we set $Q=I-P$ for a projection
$P$. We denote the Banach algebra of bounded linear operators on $\X$ equipped with its natural norm by $B(\X)$.

If $\Y$ is another Banach space, we then let $BC(\R , \X)$ (respectively, $BC(\R \times \Y, \X)$)
denote the collection of all $\X$-valued bounded continuous
functions and equip it with the sup norm (respectively, the space of jointly bounded continuous
functions $F: \R \times \Y \mapsto \X$).  

The space $BC(\R, \X)$
equipped with the sup norm is a
Banach space. Furthermore, $C(\R, \Y)$ (respectively, $C(\R \times
\Y, \X)$) denotes the class of continuous functions from $\R$ into
$\Y$ (respectively, the class of jointly continuous functions $F:
\R \times \Y \mapsto \X$).

\subsection{Evolution Families}\label{EF}
The setting of this Subsection follows that of Baroun {\it et al.} \cite{W} and Diagana \cite{TOK}. 
Fix once and for all a Banach space $(\X, \|\cdot\|)$.
\begin{definition}\label{DEF} A family of closed linear operators
$A(t)$ for $t\in \R$ on $\X$ with domain $D(A(t))$ (possibly not
densely defined) satisfy the so-called Acquistapace and Terreni
conditions, if there exist constants $\omega\in \R$, $\theta
\in (\frac{\pi}{2},\pi)$, $L > 0$ and $\mu, \nu \in (0,
1]$ with $\mu + \nu > 1$ such that
\begin{equation}\label{AT1}
  \Sigma_\theta \cup \{0\} \subset \rho(A(t)-\om) \ni \la,\;\qquad
  \|R(\la,A(t)-\om)\|\le \frac{K}{1+|\la|} \ \ \ \mbox{for all} \ t \in \R, \end{equation}
   and
\begin{equation}\label{AT2}
\|(A(t)-\om)R(\la,A(t)-\om)\,[R(\om,A(t))-R(\om,A(s))]\|
  \le L\, \frac{|t-s|^\mu}{|\la|^{\nu}}
  \end{equation}
for $t,s\in\R$, $\displaystyle \la \in\Sigma_\theta:=
\{\la\in\C\setminus\{0\}: |\arg \la|\le\theta\}$.
\end{definition}

For a given family of linear operators $A(t)$, the existence of an evolution family associated with it is not always guaranteed. However, if $A(t)$ satisfies Acquistapace-Terreni, then there exists a unique evolution
family $${\mathcal U}= \{U(t,s): t, s \in \R \ \ \mbox{such that} \ \ t \geq s\}$$ on $\X$ associated with $A(t)$ 
such that $U(t, s)\X \sub D(A(t))$ for all $t, s \in \R$ with $t \geq s$,
and

\begin{enumerate}
\item[(a)]  $U(t,s)U(s,r)=U(t,r)$ for $t,s \in \R$ such that $t \geq s \geq s$;

\item[(b)] $U(t,t)=I$ for $t \in \R$ where $I$ is the identity operator of $\X$;

\item[(c)] $(t,s)\mapsto U(t,s)\in B(\X)$ is continuous for $t>s$;

\item[(d)] $U(\cdot,s)\in C^1((s,\infty),B(\X))$, $\displaystyle
\frac{\partial U}{\partial t}(t,s) =A(t)U(t,s)$ and
\begin{align*}\label{au}
  \left\|A(t)^k U(t,s)\right\|&\le K\,(t-s)^{-k}
  \end{align*}
for $0< t-s\le 1$ and $k=0,1$.
\end{enumerate}

\begin{definition}
An evolution family ${\mathcal U} = \{U(t,s): t, s \in \R \ \ \mbox{such that} \ \ t \geq s\}$ is said to have an {\it exponential
  dichotomy} (or is {\it hyperbolic}) if there are projections
$P(t)$ ($t\in\R$) that are uniformly bounded and strongly
continuous in $t$ and constants $\delta>0$  and $N\ge1$ such that
\begin{enumerate}
\item[(e)] $U(t,s)P(s) = P(t)U(t,s)$; \item[(f)] the restriction
$U_Q(t,s):Q(s)\X\to Q(t)\X$ of $U(t,s)$ is
  invertible (we then set $\widetilde{U}_Q(s,t):=U_Q(t,s)^{-1}$); and
\item[(g)] $\left\|U(t,s)P(s)\right\| \le Ne^{-\delta (t-s)}$ and
  $\left\|\widetilde{U}_Q(s,t)Q(t)\right\|\le Ne^{-\delta (t-s)}$ for $t\ge s$ and $t,s\in \R$.
\end{enumerate}
\end{definition}

This setting requires some estimates related to ${\mathcal U} =\{U(t,s)\}_{t \geq s}$. For that, we
introduce the interpolation spaces for $A(t)$. 

Let $A$ be a sectorial operator on $\X$ (in Definition \ref{DEF}, replace
$A(t)$ with $A$) and let $\alpha\in(0,1)$. Define
the real interpolation space
$$\displaystyle \X^A_{\alpha}: = \Big\{x\in \X: \|x\|^A_{\alpha}:=
\sup\nolimits_{r>0}
\|r^{\alpha}(A-\omega)R(r,A-\omega)x\|<\infty
\Big\},
$$ which, by the way, is a Banach space when endowed with the norm $\|\cdot
\|^A_{\alpha}$. For convenience we further write
$$\X_0^A:=\X,
\ \|x\|_0^A:=\|x\|, \ \X_1^A:=D(A)$$ and
$\|x\|^A_{1}:=\|(\omega-A)x\|$. Moreover, let
$\hat{\X}^A:=\overline{D(A)}$ of $\X$.

\begin{definition}
Given a family of linear operators $A(t)$ for $t\in \R$ satisfying the Acquistapace-Terreni conditions, we set
$\X^t_\alpha:=\X_\alpha^{A(t)}$ and $
\hat{\X}^t:=\hat{\X}^{A(t)}$ for $0\le \alpha\le 1$ and
$t\in\R$, with the corresponding norms. 
\end{definition}

\begin{proposition}\cite{W}\label{pes}
For $x \in \X$, $ 0\leq \alpha \leq 1$ and $t > s,$ the following
hold:

\begin{enumerate}
\item[(i)] There is a constant $c(\alpha),$ such that 
 \begin{equation}\label{eq1.1}
  \|U(t,s)P(s)x\|_{\alpha}^t\leq
 c(\alpha)e^{- \frac{\delta}{2}(t-s)}(t-s)^{-\alpha} \|x\|.
  \end{equation}
\item[(ii)] There is a constant $m(\alpha),$ such that

 \begin{equation}\label{eq2.1}
 \|\widetilde{U}_{Q}(s,t)Q(t)x\|_{\alpha}^s\leq
 m(\alpha)e^{-\delta (t-s)}\|x\|, \qquad t \leq s.
 \end{equation}
 \end{enumerate}

\end{proposition}

\subsection{Properties of Weights} This subsection is similar to the one given in Diagana \cite{DW} except that most of all the proofs will be omitted.

Let $\U$ denote the collection of functions (weights)
$\rho: \R \mapsto (0, \infty)$, which are locally integrable over
$\R$ such that $\rho > 0$ almost everywhere.

In the rest of the paper, if $\mu \in \U$, $T
> 0$, and $a \in \R$, we then set $Q_T := [-T, T]$, $Q_T + a := [-T+a, T+a]$, and 
$$\mu(Q_T) := \int_{Q_T} \mu(x) dx.$$

Here as in the particular case when $\mu(x) = 1$ for each $x \in \R$,
we are exclusively interested in the weights $\mu$ for which,
$$\displaystyle \lim_{T \to \infty} \mu(Q_T) = \infty.$$ Consequently, we define the
space of weights $\W$ by
$$\W : = \Bigg\{ \mu \in \U: \ \inf_{x\in \R} \mu(x) = \mu_0 > 0 \ \ \mbox{and} \ \ \lim_{T \to \infty} \mu(Q_T) = \infty\Bigg\}.$$

In addition to the above, we define the set
of weights $\V$ by
$$\displaystyle \V := \Bigg\{\mu \in \W: \ \sup_{x \in \R} \mu(x) = \mu_1 < \infty \Bigg\}.$$

We also need the following set of weights, which makes the spaces of weighted pseudo-almost periodic functions translation-invariant, 

$$\displaystyle \W^{\rm Inv} := \Bigg\{\mu \in \W: \ \lim_{x \to \infty} \frac{\mu(x+\tau)}{\mu(x)} < \infty \ \ \mbox{and} \ \ \lim_{T \to \infty} \frac{\mu(Q_{T+\tau})}{\mu(Q_T)} < \infty \ \mbox{for all} \ \tau \in \R\Bigg\}.$$

Let $\W^c$ denote the collection of all continuous functions (weights) $\mu: \R \mapsto (0, \infty)$ such that $\mu > 0$ almost everywhere.

Define

$$\displaystyle \W^s := \Bigg\{\mu \in \W^c \cap \W: \ \lim_{x \to \infty} \frac{\mu(x+\tau)}{\mu(x)} < \infty \ \ \mbox{for all} \ \tau \in \R\Bigg\}.$$

\begin{lemma}\cite[Diagana]{DW}\label{TL} The inclusion $\W^s \subset \W^{\rm Inv}$ holds.
\end{lemma}

\begin{definition}
Let $\mu, \nu \in \W$. One says that $\mu$ is equivalent
to $\nu$ and denote it $\mu \prec \nu$, if $\displaystyle
\frac{\mu}{\nu} \in \V.$
\end{definition}

Let $\mu, \nu, \gamma \in \W.$ It is clear that $\mu
\prec \mu$ (reflexivity); if $\mu \prec \nu$, then
$\nu \prec \mu$ (symmetry); and if $\mu \prec \nu$ and
$\nu \prec \gamma$, then $\mu\prec \gamma$ (transitivity).
Therefore, $\prec$ is a binary equivalence relation on $\W$. 

We have

\begin{proposition}\label{TOD} Let $\mu, \nu \in \W^{\rm Inv}$. If $\mu \prec \nu$, then $\sigma = \mu + \nu \in  \W^{\rm Inv}$.
\end{proposition}

\begin{proposition}\label{TOKA} Let $\mu, \nu \in \W^s$. Then their product $\pi = \mu \nu \in  \W^s$. Moreover, if $\mu \prec \nu$, then $\sigma : = \mu + \nu \in \W^s$.

\end{proposition}

The next theorem describes all the nonconstant polynomials belonging to the set of weights $\W$.

\begin{theorem}\cite[Diagana]{DW} If $\mu \in \W$ is a nonconstant polynomial of degree $N$, then $N$ is necessarily even ($N = 2n'$ for some nonnegative integer $n'$). More precisely, $\mu$ can be written in the following form:
$$\mu(x) = a  \prod_{k=0}^{n} (x^2 + a_k x + b_k)^{m_k}$$
where $a> 0$ is a constant, $a_k$ and $b_k$ are some real numbers satisfying $a_k^2 - 4b_k < 0$, and $m_k$ are nonnegative integers for $k =0, ..., n$. Furthermore,  the weight $\mu$ given above belongs to $\W^s$.
\end{theorem}

\subsection{Doubly-Weighted Pseudo-Almost Periodic Functions}

\begin{definition}\label{D}
A function $f \in C(\R , \X)$ is called (Bohr) almost periodic if
for each $\varepsilon > 0$ there exists $l(\varepsilon)
> 0$ such that every interval of length  $l(\varepsilon)
$ contains a number $\tau$ with the property that
$$\|f(t +\tau)
- f(t) \| < \varepsilon \ \ \mbox{for each} \ \ t \in \R.$$
\end{definition}

The collection of
all almost periodic functions will be denoted $AP(\X)$.

\begin{definition}\label{D}
A function $F \in C(\R \times \Y, \X)$ is called (Bohr) almost
periodic in $t \in \R$ uniformly in $y \in \Y$ if for each
$\varepsilon
> 0$ and any compact $K \subset \Y$ there exists $l(\varepsilon)$ such that every interval of length  $l(\varepsilon)$ contains
a number $\tau$ with the property that
$$\|F(t + \tau, y) - F(t, y)\| < \varepsilon \ \ \mbox{for each} \ \ t \in \R, \ \ y \in K.$$

The collection of those functions is denoted by $AP(\Y, \X)$.
\end{definition}

If $\mu, \nu \in \W$, we then define

$$PAP_0(\X, \mu, \nu) := \Bigg\{ f \in BC(\R , \X): \ \ \lim_{T \to \infty}
\displaystyle{\frac{1}{\mu(Q_T)}} \int_{Q_T} \left\|
f(\sigma)\right\| \, \nu(\sigma) \, d\sigma = 0\Bigg\}.$$

Similarly, we define $PAP_0(\Y, \X, \mu, \nu)$ as the collection
of jointly continuous functions $F: \R \times \Y \mapsto \X$ such
that $F(\cdot, y)$ is bounded for each $y \in \Y$ and
$$\lim_{T \to \infty} \displaystyle{\frac{1}{\mu(Q_T)}}
\left\{\int_{Q_T} \| F(s, y)\| \, \nu(s) \, ds\right\} = 0$$
uniformly in $y \in \Y$.

\begin{definition}\label{DD} Let $\mu, \nu \in \W$.
A function $f \in C(\R , \X)$ is called doubly-weighted pseudo-almost
periodic if it can be expressed
as $f = g + \phi,$ where $g \in AP(\X)$ and $\phi \in PAP_0(\X,
\mu, \nu)$. The collection of such functions will be denoted by
$PAP({\mathbb X}, \mu, \nu)$.
\end{definition}

\begin{definition}\label{KK} Let $\mu, \nu \in \W$.
A function $F \in C(\R \times \Y , \X)$ is called doubly-weighted pseudo-almost
periodic if it can be expressed
as $F= G + \Phi,$ where $G \in AP(\Y, \X)$ and $\Phi \in PAP_0(\Y, \X,
\mu, \nu)$. The collection of such functions will be denoted by
$PAP(\Y, {\mathbb X}, \mu, \nu)$.
\end{definition}

\begin{proposition}\cite[Diagana]{DW}\label{P26}
Let $\mu \in \W$ and let $\nu \in \W^{\rm Inv}$ such that \begin{eqnarray}\label{HHH}\displaystyle \sup_{T > 0}\Bigg[\displaystyle{\frac{\nu(Q_{T})}{\mu(Q_T)}}\Bigg]< \infty.\end{eqnarray}Let $f \in PAP_0(\R, \mu, \nu)$ and let $g \in
L^1(\R)$. Suppose 
\begin{eqnarray}\label{JJ} \displaystyle \lim_{T \to \infty} \Bigg[\displaystyle{\frac{\mu(Q_{T+|\tau|})}{\mu(Q_T)}}\Bigg] < \infty \ \ \mbox{for all} \ \ \tau \in \R.\end{eqnarray}
Then $f \ast g$, the convolution of $f$ and $g$ on $\R$,
belongs to $PAP_0(\R, \mu, \nu)$.
\end{proposition}

\begin{proof}
It is clear that if $f \in PAP_0(\R, \mu, \nu)$ and $g \in L^1(\R)$, then their convolution
$f \ast g \in BC(\R, \R)$. Now setting $$J(T, \mu, \nu) :=
\displaystyle{\frac{1}{\mu(Q_T)}} \int_{Q_T}
\int_{-\infty}^{+\infty} |f(t-s)|\, |g(s)| \nu(t)\, ds dt$$it
follows that
\begin{eqnarray}
\displaystyle{\frac{1}{\mu(Q_T)}} \int_{Q_T} |(f
\ast g)(t)| \nu(t) dt &\leq& J(T,\mu, \nu)\nonumber \\
&=& \int_{-\infty}^{+\infty} |g(s)| \left(
\displaystyle{\frac{1}{\mu(Q_T)}}
\int_{Q_T} |f(t-s)| \nu(t) dt\right)ds \nonumber \\
&=& \int_{-\infty}^{+\infty} |g(s)| \phi_{T}(s) ds,\nonumber
\end{eqnarray} where
\begin{eqnarray*}
\displaystyle \phi_T(s) &=& \frac{1}{\mu(Q_T)}
\int_{Q_T} |f(t-s)| \nu(t) dt \\
&=& \frac{\mu(Q_{T+|s|})}{\mu(Q_T)}
\,.\, \frac{1}{\mu(Q_{T+|s|})}
\int_{Q_T} |f(t-s)| \nu(t) dt\\
&\leq& \frac{\mu(Q_{T+|s|})}{\mu(Q_T)}
\,.\, \frac{1}{\mu(Q_{T+|s|})}
\int_{Q_{T+|s|}} |f(t)| \nu(t+s) dt.
\end{eqnarray*}
Using the fact that $\nu \in \W^{\rm Inv}$ and Eq. (\ref{JJ}), one can easily see that
$\phi_T(s)
\mapsto 0$ as $T \mapsto \infty$ for all $s \in \R$. Next, since $\phi_T$ is bounded, i.e., 
$$|\phi_T(s)| \leq \|f\|_\infty \,.\,\displaystyle \sup_{T > 0}\displaystyle{\frac{\nu(Q_{T})}{\mu(Q_T)}}< \infty$$ and $g \in L^1(\R)$, using the
Lebesgue Dominated Convergence Theorem it follows that
$$\lim_{T \to \infty} \left\{\int_{-\infty}^{+\infty} |g(s)| \phi_{T}(s) ds
\right\} = 0,$$ and hence $f \ast g \in PAP_0(\R, \mu, \nu)$.  
\end{proof}

\begin{corollary} Let $\mu \in \W$ and let $\nu \in \W^{\rm Inv}$ such that Eqs. (\ref{HHH}) -- (\ref{JJ}) hold. If $f \in PAP(\R, \mu, \nu)$ and $g \in
L^1(\R)$, then $f \ast g$ belongs to $PAP (\R, \mu, \nu)$.
\end{corollary}

\begin{theorem}\cite[Diagana]{DW} If $\mu, \nu \in \W$ such that the space $PAP_0(\X, \mu, \nu)$ is translation-invariant and if
\begin{eqnarray}\label{I}\displaystyle \inf_{T > 0} \Bigg[\displaystyle{\frac{\nu(Q_{T})}{\mu(Q_T)}}\Bigg] = \delta_0 > 0,
\end{eqnarray}
then the decomposition of doubly-weighted pseudo-almost periodic functions is unique.
\end{theorem}

\begin{theorem}\cite[Diagana]{DW}\label{toka} Let $\mu, \nu \in \W$ and
let $f \in PAP(\Y, \X, \mu, \nu)$ satisfying the Lipschitz condition
$$\|f(t, u) - f(t,v)\| \leq L \,.\, \|u-v\|_{\Y} \ \ \mbox{for
all} \ \ u,v \in \Y, \ t \in \R.$$ If $h \in PAP(\Y, \mu, \nu)$,
then $f(\cdot , h(\cdot)) \in PAP(\X, \mu, \nu)$.
\end{theorem}

\section{Existence of a Doubly-Weighted Mean for Almost Periodic Functions}
Let $\mu, \nu \in \W$. If $f: \R \mapsto \X$ is a bounded continuous function, we define its {\it doubly-weighted mean}, if the limit exists, by
$${\mathcal M}(f, \mu, \nu): = \lim_{T \to \infty} \frac{1}{\mu(Q_T)} \int_{Q_T} f(t) \nu(t) dt.$$

It is well-known that if $f \in AP(\X)$, then its mean defined by
$${\mathcal M}(f):= \lim_{T\to \infty} \frac{1}{2T} \int_{Q_T} f(t)dt$$
exists \cite{B}. Consequently, for every $\lambda \in \R$, the following limit
$$a(f, \lambda):= \lim_{T \to \infty} \frac{1}{2T} \int_{Q_T} f(t) e^{-i \lambda t} dt$$
exists and is called the Bohr transform of $f$. 

It is well-known that
$a(f, \lambda)$ is nonzero at most at countably many points \cite{B}.
The set defined by
$$\sigma_b(f) := \Big\{\lambda \in \R: a(f, \lambda) \not = 0\Big\}$$
is called the Bohr spectrum of $f$ \cite{M}. 


\begin{theorem}\label{H} {\rm (Approximation Theorem)} \cite{LV, M} Let $f \in AP(\X)$. Then for every $\varepsilon > 0$ there exists a trigonometric polynomial 
$$P_\varepsilon (t) = \sum_{k=1}^n a_k e^{i\lambda_k t}$$ where $a_k \in \X$ and $\lambda_k \in \sigma_b (f)$ such that
$\|f(t) - P_\varepsilon (t)\| < \varepsilon$
for all $t \in \R$.

\end{theorem}

In Liang {\it et al.} \cite{liang},
the original question which consists of the existence of a weighted mean for almost periodic functions was raised. In particular, Liang {\it et al.} have shown through an example that there exist weights for which a weighted mean for almost periodic functions may not exist. In this section we investigate the broader question, which consists of the existence of a doubly-weighted mean for almost periodic functions. Namely, we give some sufficient conditions, which do guarantee the existence of a doubly-weighted mean for almost periodic functions. Moreover, under those conditions, it will be shown that the doubly-weighted mean and the classical (Bohr) mean are proportional (Theorem \ref{XX}). Further, it will be shown that if the doubly-weighted Bohr spectrum of an almost periodic function exists, then it is either empty or coincides with the Bohr spectrum of that function.

We have

\begin{theorem}\label{XX} Let $\mu, \nu \in \W$ and suppose that $\displaystyle \lim_{T \to \infty} \frac{\nu(Q_T)}{\mu(Q_T)} = \theta_{\mu\nu}$. If $f: \R \mapsto \X$ is an almost periodic function such that
\begin{eqnarray}\label{CD}
\lim_{T \to \infty} \Bigg|\frac{1}{\mu(Q_T)} \int_{Q_T} e^{i\lambda t} \nu(t) dt\Bigg| = 0
\end{eqnarray}
for all $0 \not= \lambda \in \sigma_b(f)$, then the doubly-weighted mean of $f$,

$${\mathcal M}(f, \mu, \nu) = \lim_{T \to \infty} \frac{1}{\mu(Q_T)} \int_{Q_T} f(t) \nu(t) dt$$ exists. Furthermore, ${\mathcal M}(f, \mu, \nu) = \theta_{\mu\nu} {\mathcal M}(f)$.
\end{theorem}

\begin{proof} If $f$ is a trigonometric polynomial, say, $\displaystyle f(t) = \sum_{k=0}^n a_k e^{i\lambda_k t}$
where $a_k \in \X-\{0\}$ and $\lambda_k \in \R$ for $k = 1, 2, ..., n$, then $\sigma_b(f) = \{\lambda_k: \ k =1, 2, ..., n\}$. Moreover, 
\begin{eqnarray*}\frac{1}{\mu(Q_T)} \int_{Q_T} f(t) \nu(t) dt &=& a_0 \frac{\nu(Q_T)}{\mu(Q_T)} + \frac{1}{\mu(Q_T)} \int_{Q_T} \Big[\sum_{k=1}^n a_k e^{i\lambda_k t}\Big] \nu(t) dt\\
&=&a_0 \frac{\nu(Q_T)}{\mu(Q_T)} + \sum_{k=1}^n a_k \Big[\frac{1}{\mu(Q_T)}\int_{Q_T} e^{i\lambda_k t} \nu(t) dt\Big]\\
\end{eqnarray*} 
and hence
\begin{eqnarray*} \left\|\frac{1}{\mu(Q_T)} \int_{Q_T} f(t) \nu(t) dt - a_0 \frac{\nu(Q_T)}{\mu(Q_T)}\right\| &\leq& \sum_{k=1}^n \left\|a_k\right\|  \Big|\frac{1}{\mu(Q_T)} \int_{Q_T} e^{i\lambda_k t} \nu(t) dt\Big|\\
\end{eqnarray*}
which by Eq. (\ref{CD}) yields 
$$\left\|\frac{1}{\mu(Q_T)} \int_{Q_T} f(t) \nu(t) dt - a_0 \theta_{\mu\nu}\right\| \to 0 \ \ \mbox{as} \ \ T \to \infty$$and therefore ${\mathcal M}(f, \mu, \nu) = a_0 \theta_{\mu\nu} = \theta_{\mu\nu} M(f)$.

If in the finite sequence of $\lambda_k$ there exist $\lambda_{n_k} = 0$ for $k = 1, 2, ...l$ with $a_m \in \X-\{0\}$ for all $m \not = n_k$ ($k=1,2,...,l$), it can be easily shown that $$\displaystyle {\mathcal M}(f, \mu, \nu) = \theta_{\mu\nu} \sum_{k=1}^l a_{n_k} = \theta_{\mu\nu} M(f).$$

Now if $f: \R \mapsto \X$ is an arbitrary almost periodic function, then for every $\varepsilon > 0$ there exists a trigonometric polynomial (Theorem \ref{H}) $P_\varepsilon$ defined by
$$P_\varepsilon (t) = \sum_{k=1}^n a_k e^{i\lambda_k t}$$ where $a_k \in \X$ and $\lambda_k \in \sigma_b (f)$ such that
\begin{eqnarray}\label{11} \left\|f(t) - P_\varepsilon (t)\right\| < \varepsilon\end{eqnarray}
for all $t \in \R$.

Proceeding as in Bohr \cite{B} it follows that there exists $T_0$ such that for all $T_1, T_2 > T_0$,
\begin{eqnarray*}\label{12}
\Big\|\frac{1}{\mu(Q_{T_1})} \int_{Q_{T_1}} P_\varepsilon(t) \nu(t) dt - \frac{1}{\mu(Q_{T_2})} \int_{Q_{T_2}} P_\varepsilon(t) \nu(t) dt  \Big\| = \theta_{\mu\nu} \Big\|M(P_\varepsilon) - M(P_\varepsilon)\Big\| = 0 < \varepsilon.
\end{eqnarray*}

In view of the above it follows that for all $T_1, T_2 > T_0$,

\begin{eqnarray*}
\noindent\Big\|\frac{1}{\mu(Q_{T_1})} \int_{Q_{T_1}} f(t) \nu(t) dt - \frac{1}{\mu(Q_{T_2})} \int_{Q_{T_2}} f(t) \nu(t) dt \Big\| &\leq& \frac{1}{\mu(Q_{T_1})} \int_{Q_{T_1}} \| f(t) - P_\varepsilon(t)\| \nu(t) dt\\
+ \Big\|\frac{1}{\mu(Q_{T_1})} \int_{Q_{T_1}} P_\varepsilon(t) \nu(t) dt - \frac{1}{\mu(Q_{T_2})} \int_{Q_{T_2}} P_\varepsilon(t) \nu(t) dt  \Big\|\\
+ \frac{1}{\mu(Q_{T_2})} \int_{Q_{T_2}} \| f(t) - P_\varepsilon(t)\| \nu(t) dt
<3\varepsilon.
\end{eqnarray*}
\end{proof}

\begin{example} Fix a natural number $N > 1$. Let $\mu (t) = e^{|t|}$ and $\nu(t) = (1 + |t|)^N$ for all $t \in \R$, which yields $\theta_{\mu\nu} = 0$. If $\varphi: \R \mapsto \X$ is a (nonconstant) almost periodic function, then according to the previous theorem,
its doubly-weighted mean
${\mathcal M}(\varphi, \mu, \nu)$ exists. Moreover,

$$\lim_{T \to \infty} \frac{1}{2(e^T - 1)} \int_{Q_T} f(t) (1 + |t|)^N dt = 0 . \lim_{T \to \infty} \frac{1}{2T} \int_{Q_T} f(t) dt = 0.$$
\end{example}

Consider the set of weights $\W^{0}$ defined by
$$\W^{0} = \Bigg\{\mu \in \W: D_\tau := \lim_{|t| \to \infty} \frac{\mu(Q_{t+\tau})}{\mu(Q_t)}
< \infty \ \ \mbox{for all} \ \ \tau \in \R \Bigg\}.$$

Setting,
$\displaystyle C_\tau =  \lim_{|t| \to \infty} \frac{\mu(Q_{t}+\tau)}{\mu(Q_t)}$, 
one can easily see that $C_\tau \leq D_\tau < \infty$ for all $\tau \in \R$.

\begin{corollary}\label{X} Fix $\mu, \nu \in \W^{0}$ and suppose that $\displaystyle \lim_{T \to \infty} \frac{\nu(Q_T)}{\mu(Q_T)} = \theta_{\mu\nu}$. If $f: \R \mapsto \X$ is an almost periodic function such that Eq. (\ref{CD}) holds, then 
\begin{eqnarray}\label{IN}
{\mathcal M} (f_a, \mu, \nu_a) = C_{-a} \theta_{\mu\nu} {\mathcal M} (f) = C_{-a} {\mathcal M}(f, \mu, \nu)
\end{eqnarray}
uniformly in $a \in \R$, where
$${\mathcal M} (f_b, \mu, \nu_b) = \lim_{T \to \infty} \frac{1}{\mu(Q_T)} \int_{Q_T} f_b(t) \nu_b(t) dt =
\lim_{T\to \infty} \frac{1}{\mu(Q_T)} \int_{Q_T} f(t+b) \nu(t+b) dt$$
for each $b \in \R$. 
\end{corollary}

\begin{proof} Clearly, the existence of ${\mathcal M}(f, \mu, \nu)$ is guaranteed by Theorem \ref{XX}. Without lost of generality, suppose $a > 0$. Now since $f \in AP(\X)$ it follows that $f_a: t \mapsto f(t + a)$ belongs to $AP(\X)$. Moreover, the weight $\nu_a$ defined by $\nu_a(t) = \nu(t+a)$ for all $t \in \R$ belongs to $\W^0$. 

Now

$$\Big| \int_{Q_T} e^{i \lambda t} \nu_a (t) dt \Big| = \Big| \int_{Q_T-a} e^{i \lambda (t-a)} \nu (t) dt \Big| = \Big| \int_{Q_T-a} e^{i \lambda t} \nu (t) dt \Big| \leq \Big| \int_{Q_{T+a}} e^{i \lambda t} \nu (t) dt \Big|$$
and hence 

\begin{eqnarray*}
\lim_{T \to \infty} \Big| \frac{1}{\mu(Q_T)}\int_{Q_T} e^{i \lambda t} \nu_a (t) dt \Big| &=& \lim_{T\to \infty} \Big|\frac{1}{\mu(Q_{T})} \int_{Q_{T}-a} e^{i \lambda t} \nu (t) dt \Big| \\
&\leq&  \lim_{T\to \infty} \Big|\frac{1}{\mu(Q_{T})} \int_{Q_{T+a}} e^{i \lambda t} \nu (t) dt \Big| \\
&=& \lim_{T\to \infty} \Big|\frac{\mu(Q_{T+a})}{\mu(Q_{T})}\frac{1}{\mu(Q_{T+a})} \int_{Q_{T+a}} e^{i \lambda t} \nu (t) dt \Big| \\
&=& D_a\lim_{T\to \infty} \Big|\frac{1}{\mu(Q_{T+a})} \int_{Q_{T+a}} e^{i \lambda t} \nu (t) dt \Big| \\
&=&0.
\end{eqnarray*}

Now

$$\displaystyle \lim_{T \to \infty} \frac{\nu_a (Q_T)}{\mu(Q_T)} = C_{-a} \theta_{\mu\nu}.$$

Using Theorem \ref{XX} it follows that for every $\varphi \in AP(\X)$,

$${\mathcal M}(\varphi_a, \mu, \nu_a) = \lim_{T \to \infty} \frac{1}{\mu(Q_T)} \int_{Q_T} \varphi_a(t) \nu_a(t) dt$$ exists. Furthermore, ${\mathcal M}(\varphi_a, \mu, \nu_a) = C_{-a} \theta_{\mu\nu} {\mathcal M}(\varphi_a)$ for all $a \in \R$. In particular, ${\mathcal M}(f_a, \mu, \nu_a) = C_{-a}  \theta_{\mu\nu} {\mathcal M}(f_a)$ uniformly in $a \in \R$. Now from Bohr \cite{B}, ${\mathcal M}(f_a) =  {\mathcal M}(f)$ uniformly in $a \in \R$, which completes the proof.

\end{proof}

\begin{definition}
Fix $\mu, \nu \in \W$ and suppose that $\displaystyle \lim_{T \to \infty} \frac{\nu(Q_T)}{\mu(Q_T)} = \theta_{\mu\nu}$. If $f: \R \mapsto \X$ is an almost periodic function such that Eq. (\ref{CD}) holds, we then define its doubly-weighted Bohr transform as
$$\widehat{a}_{\mu\nu}(f)(\lambda):= \lim_{T \to \infty} \frac{1}{\mu(Q_T)} \int_{Q_T} f(t) e^{-i \lambda t} \nu(t) dt \ \ \mbox{for all} \ \ \lambda \in \R.$$
\end{definition}

Now since $t \mapsto g_\lambda(t):= f(t) e^{-i \lambda t} \in AP(\X)$ it follows that
$$\widehat{a}_{\mu\nu}(f)(\lambda) = \theta_{\mu\nu} {\mathcal M} (f(\cdot) e^{-i \lambda \cdot}) = \theta_{\mu\nu} a(f, \lambda).$$ That is, under Eq. (\ref{CD}), 
$$\widehat{a}_{\mu \nu}(f)(\lambda):= \lim_{T \to \infty} \frac{1}{\mu(Q_T)} \int_{Q_T} f(t) e^{-i \lambda t} \nu(t) dt = \theta_{\mu\nu} \lim_{T \to \infty} \frac{1}{2T} \int_{Q_T} f(t) e^{-i \omega t} dt = \theta_{\mu\nu} a(f, \lambda)$$
for all $\lambda \in \R.$

In summary, there are two possibilities for the doubly-weighted Bohr spectrum of an almost periodic function. Indeed, 

1) If $\displaystyle \lim_{T \to \infty} \frac{\nu(Q_T)}{\mu(Q_T)} = \theta_{\mu\nu} = 0$, then 
$\widehat{a}_{\mu\nu}(f)(\lambda) = \theta_{\mu\nu}a(f, \lambda) = 0$ for all $\lambda \in \R$. In that event, the doubly-weighted Bohr spectrum of $f$ is
$$\sigma_b^{\mu\nu}(f) := \Big\{\lambda \in \R: \widehat{a}_{\mu\nu}(f)(\lambda) \not = 0\Big\} = \emptyset.$$

2) If $\displaystyle \lim_{T \to \infty} \frac{\nu(Q_T)}{\mu(Q_T)} = \theta_{\mu\nu} \not = 0$, then 
$\widehat{a}_{\mu\nu}(f)(\lambda) = \theta_{\mu\nu}a(f, \lambda)$ exists for all $\lambda \in \R$ and is nonzero at most at countably many points. 
In that event, the doubly-weighted Bohr spectrum of $f$ is 
$$\sigma_b^{\mu\nu}(f) := \Big\{\lambda \in \R: \widehat{a}_{\mu\nu}(f)(\lambda) \not = 0\Big\} = \Big\{\lambda \in \R: a(f, \lambda) \not = 0\Big\},$$
that is, 
$\sigma_b^{\mu\nu}(f)= \sigma_b(f).$
In particular, $\sigma_b^{\mu\mu}(f)= \sigma_b(f).$

\section{Existence of Doubly-Weighted Pseudo-Almost Periodic Solutions to Some
Differential Equations} 
In this Section, we fix two weights $\mu, \nu \in \W$ such that $PAP(\X, \mu, \nu)$ is translation-invariant and Eq. (\ref{I}) holds. Under these assumptions, it can be easily shown that
$PAP(\X, \mu, \nu)$ is a Banach space when equipped with the sup norm.

In what follows, we
denote by $\Gamma_1$ and $\Gamma_2$, the nonlinear integral operators
defined by
$$(\Gamma_1 u)(t) :=\int_{-\infty}^{t}U(t,s)P(s) g(s, 
u(s))ds, \ \mbox{and}$$ and $$ (\Gamma_2 u)(t)
:=\int_{t}^{\infty}U_Q(t,s)Q(s) g(s, u(s))ds .$$

To study the existence of doubly-weighted pseudo-almost periodic solutions to Eq. (\ref{2}) we will assume that the following assumptions hold:

\begin{enumerate}
  \item [(H.1)] The family of closed linear operators
$A(t)$ for $t\in \R$ on $\X$ with domain $D(A(t))$ (possibly not
densely defined) satisfy Acquistapace and Terreni
conditions, that is, there exist constants $\omega\in \R$, $\theta
\in \Big(\frac{\pi}{2},\pi\Big)$, $L > 0$ and $\mu, \nu \in (0,
1]$ with $\mu + \nu > 1$ such that
\begin{equation*}\label{AT1}
  \Sigma_\theta \cup \Big\{0\Big\} \subset \rho\Big(A(t)-\omega\Big) \ni \lambda,\;\qquad
  \\ \|R(\lambda,A(t)-\omega)\|\le \frac{K}{1+|\lambda|} \ \ \ \mbox{for all} \ t \in \R, \end{equation*}
   and
\begin{equation*}\label{AT2}
\|(A(t)-\omega)R(\lambda,A(t)-\omega)\,[R(\omega,A(t))-R(\omega,A(s))]\|
  \leq L\, \frac{|t-s|^\mu}{|\lambda|^{\nu}}
  \end{equation*}
for $t,s\in\R$, $\displaystyle \lambda \in\Sigma_\theta:=
\{\lambda\in\C \setminus\{0\}: |\arg \lambda|\le\theta\}$.

\item[(H.2)] The evolution family ${\mathcal U} =\{U(t,s)\}_{t \geq s}$ generated by
$A(\cdot)$ has an exponential dichotomy with constants
$N,\delta>0$ and dichotomy projections $P(t)$ for $t\in\R$.

\item[(H.3)] There exists  $0\leq \alpha<1$ such
that
$$\X_\alpha^t=\X_\alpha$$
for all $t\in \R,$ with uniform equivalent norms.

\item[(H.4)] $R(\omega, A(\cdot))  \in
AP(B(\X_\alpha))$. 

\item[(H.5)] The function $g: \R
\times \X \mapsto \X$ belongs to $PAP(\X, \X, \mu, \nu)$. Moreover, the
functions $g$ are uniformly
    Lipschitz with respect to the second argument in the following
    sense: there exists $K > 0$ such that
$$\|g(t,u)-g(t,v)\|\leq K \|u-v\|
    $$
    for all $u,v\in \X$ and $t\in \R$.

\end{enumerate}

If $0 < \alpha<1$, then the nonnegative constant $k$ will denote the 
bounds of the embedding  $\X_\alpha \hookrightarrow \X$, that is,
$$\|x\| \leq k \|x\|_\alpha$$
for all $x \in \X_\alpha$.

To study the existence and uniqueness of doubly-weighted pseudo-almost periodic
solutions to Eq. (\ref{2}) we first introduce the notion of mild
solution.

\begin{definition}
A continuous function $u: \R \mapsto \X_\alpha$ is said to be a mild solution
to Eq. (\ref{2}) if
\begin{eqnarray}
u(t)=U(t,s)u(s) + \int_{s}^{t}U(t,s)P(s) g(s, u(s)) ds -
\int_{t}^{s}U(t,s)Q(s) g(s, u(s))ds\nonumber
\end{eqnarray}
for $t \geq s$ and for all $t, s \in \R$.
\end{definition}

Under previous assumptions (H.1)-(H.5), it can be easily shown
Eq. (\ref{2}) has a unique mild solution given by
\begin{eqnarray}
u(t)=\int_{-\infty}^{t}U(t,s)P(s) g(s, u(s))ds 
- \int_{t}^{\infty}U_Q(t,s)Q(s) g(s, u(s))ds\nonumber
\end{eqnarray}
for each $t \in \R$.

\begin{lemma}\label{ll2} Under assumptions {\rm (H.1)---(H.5)}, the integral operators $\Gamma_1$ and $\Gamma_2$ defined
above map $PAP(\X_\alpha, \mu, \nu)$ into itself.
\end{lemma}

\proof Let $u \in PAP(\X_\alpha, \mu, \nu)$. Setting $h(t) = g(t, u(t))$
and using the theorem of composition of doubly-weighted pseudo-almost periodic
functions (Theorem \ref{toka}) it follows that $h \in PAP(\X, \mu, \nu)$. Now write $h
= \phi + \zeta$ where $\phi \in AP(\X)$ and $\zeta \in PAP_0(\X, \mu, \nu)$.
The nonlinear integral operator $\Gamma_1 u$ can be rewritten as
\begin{eqnarray*}
(\Gamma_1u)(t) &=&\int_{-\infty }^t U(t, s)P(s)\phi(s)ds +
\int_{-\infty }^tU(t,s)P(s)\zeta(s)ds.
\end{eqnarray*}

Set $$\displaystyle \Phi(t) = \int_{-\infty }^t
U(t,s)P(s)\phi(s)ds$$and$$\displaystyle \Psi(t)=
\int_{-\infty }^t U(t,s)P(s)\zeta(s)ds$$ for each $t \in \R$.

The next step consists of showing that $\Phi \in AP(\X_\alpha)$
and $\Psi \in PAP_0(\X_\alpha, \mu, \nu)$. Obviously, $\Phi \in
AP(\X_\alpha)$. Indeed, since $\phi \in AP(\X)$, for every
$\varepsilon > 0$ there exists $l(\varepsilon)
> 0$ such that for every interval of length $l(\varepsilon)$ contains a $\tau$
with the property
$$\|\phi(t +\tau) - \phi(t) \| < \varepsilon C\ \ \mbox{for each} \ t \in \R,$$
where $\displaystyle C = \frac{\delta^{1-\alpha}}{c(\alpha)
2^{1-\alpha} \Gamma(1 - \alpha)}$ with $\Gamma$ being the
classical $\Gamma$ function.

Now

\begin{eqnarray}
\Phi(t + \tau) - \Phi(t) &=& \int_{-\infty}^{t+\tau} U(t +\tau, s)
P(s) \phi(s)ds -  \int_{-\infty}^{t} U(t, s) P(s) \phi(s) ds\nonumber\\
&=& \int_{-\infty}^{t} U(t +\tau, s+\tau) P(s + \tau) \phi(s
+\tau)ds -  \int_{-\infty}^{t} U(t, s) P(s) \phi(s)
ds\nonumber\\
&=& \int_{-\infty}^{t} U(t +\tau, s+\tau) P(s + \tau) \phi(s
+\tau)ds \nonumber\\
&-& \int_{-\infty}^{t} U(t+\tau, s+\tau) P(s+\tau) \phi(s)
ds\nonumber\\
&+& \int_{-\infty}^{t} U(t +\tau, s+\tau) P(s+\tau) \phi(s)ds -
\int_{-\infty}^{t} U(t, s) P(s) \phi(s) ds\nonumber\\
&=& \int_{-\infty}^{t} U(t +\tau, s+\tau) P(s + \tau) \Big(\phi(s
+\tau) - \phi(s)\Big)ds\nonumber\\
&+& \int_{-\infty}^{t} \Big(U(t +\tau, s+\tau)P(s+\tau) -
U(t,s)P(s)\Big)  \phi(s)ds. \nonumber
\end{eqnarray}
Using \cite{bar, Man-Schn} it follows that
$$\left\|\int_{-\infty}^{t} \Big[U(t +\tau, s+\tau) P(s+\tau) - U(t,s)P(s)\Big]
\phi(s)ds\right\|_\alpha \leq \frac{2\|\phi\|_\infty}{\delta}
\varepsilon.$$ Similarly, using \eqref{eq1.1}, it follows that
$$\left\| \int_{-\infty}^{t} U(t +\tau, s+\tau) P(s+\tau) (\phi(s
+\tau) - \phi(s))ds\right\|_\alpha \leq \varepsilon.$$
Therefore,
$$\| \Phi(t + \tau) - \Phi(t) \|_\alpha < \Big(1 + \frac{2\|\phi\|_\infty}{\delta}\Big) \varepsilon \ \
\mbox{for each} \ t \in \R,$$and hence, $\Phi \in AP(\X_\alpha)$.

To complete the proof for $\Gamma_1$, we have to show that $\Psi
\in PAP_0(\X_\alpha, \mu, \nu)$. First, note that $s \mapsto \Psi(s)$ is a
bounded continuous function. It remains to show that
$$\lim_{T\to \infty} \displaystyle{\frac{1}{\mu(Q_T)}} \
\int_{Q_T} \| \Psi(t) \|_\alpha \nu(t)dt = 0.$$

Again using Eq. \eqref{eq1.1} it follows that

\begin{eqnarray} \displaystyle \lim_{T \to \infty} \displaystyle{\frac{1}{\mu(Q_T)}}
\ \int_{Q_T} \| \Psi(t) \|_\alpha \nu(t)dt &\leq& \lim_{T \to
\infty} \displaystyle{\frac{c(\alpha)}{\mu(Q_T)}} \ \int_{Q_T}
\int_{0}^{+\infty} s^{-\alpha} e^{-\frac{\delta}{2}s}\|
\zeta(t -s) \| \nu(t) ds dt\nonumber\\
&\leq& \lim_{T \to \infty} \displaystyle c(\alpha)
\int_{0}^{+\infty} s^{-\alpha} e^{-\frac{\delta}{2} s}
\frac{1}{\mu(Q_T)} \int_{Q_T}\| \zeta(t-s)\| \nu(t) dt ds. \nonumber
\end{eqnarray}

Set $$\displaystyle \Gamma_s(T) = \frac{1}{\mu(Q_T)} \int_{Q_T} \|
\zeta(t-s) \| \nu(t)dt.$$ Since $PAP_0(\X, \mu, \nu)$ is assumed to be translation invariant and that Eq. (\ref{I}) holds,
it follows that $t \mapsto \zeta(t-s)$ belongs to $PAP_0(\X, \mu, \nu)$ for
each $s \in \R$, and hence
$$\lim_{T \mapsto \infty} \frac{1}{\mu(Q_T)} \int_{Q_T}
\| \zeta(t-s) \| \nu(t) dt = 0$$ for each $s \in \R$.

One completes the proof by using the well-known Lebesgue Dominated
Convergence Theorem and the fact $\Gamma_s(T) \mapsto 0$ as $T \to
\infty$ for each $s \in \R$.

The proof for $\Gamma_2u(\cdot)$ is similar to that of $\Gamma_1
u(\cdot)$. However one makes use of Eq. \eqref{eq2.1} rather than
Eq. \eqref{eq1.1}.
\endproof

\begin{theorem}\label{theo}Under assumptions {\rm (H.1)---(H.5)}, then Eq. \eqref{2} has
a unique doubly-weighted pseudo-almost periodic mild solution whenever $K$ is
small enough.
\end{theorem}

\proof Consider the nonlinear operator ${\mathbb M}$ defined on
$PAP(\X_\alpha, \mu, \nu)$ by
\begin{eqnarray*}
{\mathbb M} u(t)=
\int_{-\infty}^{t}U(t,s)P(s) g(s, u(s))ds 
-\int_{t}^{\infty}U_Q(t,s)Q(s) g(s,  u(s))ds
\end{eqnarray*}
for each $t \in \R$.

In view of Lemma \ref{ll2},
it follows that ${\mathbb M}$ maps $PAP(\X_\alpha, \mu, \nu)$ into itself.
To complete the proof one has to show that ${\mathbb M}$ has a
unique fixed-point.

If $v,w\in PAP(\X_\alpha, \mu, \nu)$, then

\begin{eqnarray}
\| \Gamma_1 (v)(t) - \Gamma_1(w)(t)\|_\alpha &\leq&
\int_{-\infty}^t
\|U(t,s)P(s)\left[g(s,v(s))-g(s,w(s))\right]\|_\alpha
ds
\nonumber \\
&\leq& \int_{-\infty}^t c(\alpha) (t-s)^{-\alpha}
e^{-\frac{\delta}{2} (t-s)}\|g(s,v(s))-g(s,w(s))\|
ds
\nonumber \\
&\leq& K c(\alpha) \int_{-\infty}^t (t-s)^{-\alpha}
e^{-\frac{\delta}{2} (t-s)}\|v(s)-w(s)\| ds
\nonumber \\
&\leq& k K c(\alpha) \int_{-\infty}^t (t-s)^{-\alpha}
e^{-\frac{\delta}{2} (t-s)}\|v(s)-w(s)\|_\alpha ds
\nonumber \\
&\leq& k K c(\alpha) 2^{1-\alpha}\,\Gamma(1-\alpha)
\delta^{\alpha-1} \|v - w\|_{\alpha, \infty}, \nonumber
\end{eqnarray}
and

\begin{eqnarray}
\| \Gamma_2 (v)(t) - \Gamma_2(w)(t)\|_\alpha &\leq&
\int_{t}^{\infty} \|U_Q(t,s)Q(s)
\left[g(s,v(s))-g(s,w(s))\right]\|_\alpha ds
\nonumber \\
&\leq& \int_{t}^{\infty} m(\alpha)
e^{\delta(t-s)}\|g(s,v(s))-g(s,w(s))\| ds
\nonumber \\
&\leq& \int_{t}^{\infty} m(\alpha) K
e^{\delta(t-s)}\|v(s)-w(s)\| ds
\nonumber \\
&\leq&  k m(\alpha) K \int_{t}^{\infty}
e^{\delta(t-s)}\|v(s)-w(s)\|_\alpha ds
\nonumber \\
&\leq& K k m(\alpha) \|v - w\|_{\alpha, \infty} \int_{t}^{+\infty} e^{\delta(t-s)} ds\nonumber \\
&=& K  k m(\alpha) \delta^{-1}\|v -
w\|_{\alpha, \infty},\nonumber
\end{eqnarray}
where $\displaystyle \left\|u\right\|_{\alpha, \infty} := \sup_{t\in \R} \left\|u(t)\right\|_\alpha.$

Combining previous approximations it follows that
$$\|\mathbb{M}v-\mathbb{M}w\|_{\infty, \alpha}\leq K C(\alpha, \delta)
\,.\,\|v -
    w\|_{\alpha, \infty},$$
where $C(\alpha, \delta) = k m(\alpha) \delta^{-1} + k c(\alpha) 2^{1-\alpha}\,\Gamma(1-\alpha)
\delta^{\alpha-1} >0$ is a constant, 
and hence if the Lipschitz $K$ is small enough, then Eq. (\ref{2}) has a unique solution, which obviously is its
only doubly-weighted pseudo-almost periodic mild solution.
\endproof



\bibliographystyle{srtnumbered}
\bibliography{mybib}


\end{document}